\documentclass[VANCOUVER,STIX1COL]{WileyNJD-v2}
\usepackage{moreverb}
\usepackage{subfigure}

\usepackage{graphicx}
\newtheorem{alg}{Algorithm}

\newcommand\BibTeX{{\rmfamily B\kern-.05em \textsc{i\kern-.025em b}\kern-.08em
T\kern-.1667em\lower.7ex\hbox{E}\kern-.125emX}}

\articletype{Article Type}%

\received{<day> <Month>, <year>}
\revised{<day> <Month>, <year>}
\accepted{<day> <Month>, <year>}

\begin{document}

\title{ Greedy Motzkin-Kaczmarz methods for solving linear systems\protect}

\author{Yanjun Zhang}

\author{Hanyu Li*}

\authormark{ZHANG AND LI}

\address{\orgdiv{College of Mathematics and Statistics}, \orgname{Chongqing University}, \orgaddress{\state{Chongqing}, \country{China}}}

\corres{*Hanyu Li, College of Mathematics and Statistics, Chongqing University, Chongqing 401331, P.R. China.\\ \email{lihy.hy@gmail.com or hyli@cqu.edu.cn.}}

\presentaddress{National Natural Science Foundation of China, Grant/Award Number: 11671060; Natural Science Foundation Project of CQ CSTC, Grant/Award Number: cstc2019jcyj-msxmX0267}

\abstract[Summary]{The famous greedy randomized Kaczmarz (GRK) method uses the greedy selection rule on maximum distance to determine a subset of the indices of working rows. In this paper, with the greedy selection rule on maximum residual, we propose the greedy randomized Motzkin-Kaczmarz (GRMK) method for linear systems. The block version of the new method is also presented. We analyze the convergence of the two methods and provide the corresponding convergence factors. Extensive numerical experiments show that the GRMK method has almost the same performance as the GRK method for dense matrices and the former performs better in computing time for some sparse matrices, and the block versions of the GRMK and GRK methods always have almost the same performance.}



\maketitle


\section{Introduction}\label{sec1}
We consider the following consistent linear systems
\begin{equation}
\label{1}
Ax=b,
\end{equation}
where $A\in R^{m\times n}$, $b\in R^{m}$, and $x$ is the $n$-dimensional unknown vector. As we know, the Kaczmarz method \cite{kaczmarz1} is a popular so-called row-action method for solving the systems (\ref{1}). Its update formula is 
\begin{align}
\label{Kacz}
  x_{k+1}=x_{k}+\frac{ b^{(i)}-A^{(i)}x_k}{ \| A^{\left(i\right)} \|_{2}^{2}}( A^{(i)})^{T},
\end{align}
where $A^{(i)}$ denotes the $i$-th row of $A$, $b^{(i)}$  denotes the $i$-th entry of $b$, and $A^T$  denotes the transpose of $A$.
In 2009, 
Strohmer and Vershynin \cite{Strohmer2009} show that the Kaczmarz method converges with expected exponential rate if the row of $A$ in iteration is chosen randomly with probability proportional to the square of the Euclidean norm of the row. 
Subsequently, many randomized Kaczmarz type methods were proposed for different possible systems settings; see for example \cite{Needell2010, Completion2013, Completion2015, Dukui2019} and references therein.  
These randomized methods have two obvious disadvantages. The first one is that the probability criterion will be equivalent to the uniform sampling if the Euclidean norms of all the rows of the matrix $A$ are the same. The case can happen by 
scaling the matrix $A$ with a suitable diagonal matrix. 
The second one is that it is possible to sample the same row twice in iteration. In this case, 
no progress is made in such an update. 
To tackle these 
problems, in 2018, Bai and Wu \cite{Bai2018} constructed a greedy randomized Kaczmarz (GRK) method by introducing a more efficient probability criterion for selecting the working rows from the matrix $A$. The GRK method outperforms the ordinary randomized Kaczmarz methods in terms of the number of iterations and computing time, and 
the scheme in this method is very powerful in achieving efficient methods for solving linear problems \cite{Bai2018r,Zhang2019,huang2020remarks}, least squares problem \cite{Bai2019,Zhang2020} and ridge regression problem\cite{Liu2019}.

The greedy selection rule used in the GRK method is from the well known maximum distance rule because 
the index subset in the method is built on the combination of the maximum and average distances. 
As we know, there are two main famous greedy selection rules: the maximum distance rule and the maximum residual rule.
Specifically, let $x_{\star}=A^{\dag}b$ be the least-Euclidean-norm solution of the systems (\ref{1}). Then a sequence of vectors $x_{0}$, $x_{1}$, \ldots produced by the iteration (\ref{Kacz}) is said to converge in square to the solution $x_{\star}$ if and only if $\|x_{k}-x_\star\|^{2}_2\rightarrow 0$ as $k\rightarrow\infty$. Since the projections in iteration are orthogonal, we can check that (see also the proof of Theorem \ref{theorem1} below)
\begin{equation*}
\label{7}
\|x_{k+1}-x_{\star}\|^2_2=\|x_{k}-x_{\star}\|^2_2-\|x_{k+1}-x_{k}\|^2_2.
\end{equation*}
Hence, the optimal projection is the one that maximizes the distances $\|x_{k+1}-x_{k}\|^2_2$. 
Note that the update formula (\ref{Kacz}) implies
$$\|x_{k+1}-x_{k}\|^2_2=\left\|\frac{ b^{(i)}-A^{(i)}x_{k}}{ \| A^{ (i )} \|_{2}^{2}}( A^{(i)})^{T}\right\|^2_2,$$
which shows that  in iteration we should select the $ t_k$-th index according to
\begin{equation*}
t_k={\rm arg} \max \limits _{i }\left\|\frac{ b^{(i)}-A^{(i)}x_{k}}{ \| A^{ (i )} \|_{2}^{2}}( A^{(i)})^{T}\right\|^2_2={\rm arg} \max \limits _{i }\frac{ ( b^{(i)}- A^{(i)} x_{k})^2}{\| A^{(i)}\|^2_2}.
\end{equation*}
This greedy selection rule is the maximum distance rule \cite{Eldar2011,Gao2019,nutini2016convergence}.
The maximum residual rule \cite{griebel2012greedy,nutini2016convergence} selects the $ t_k$-th index according to 
$$t_k={\rm arg} \max \limits _{i } ( b^{(i)}- A^{(i)} x_{k})^2.$$
That is, it grasps the index corresponding to the largest magnitude entry of the residual vector $r_k=b-Ax_k$, and hence 
the largest magnitude entry of the residual vector $r_k$ can be preferentially annihilated as far as possible and make the $t_k$-th equation  be `furthest' from being satisfied. The maximum residual rule is also known as the Motzkin method \cite{Agamon54,Motzkin54}, which can also make sure that 
the same index will not be chosen twice in iteration 
and hence has 
better convergence rate compared with the ordinary randomized Kaczmarz methods. Consequently, 
many analyses and applications about Motzkin type methods were published in recent years; see for example \cite{Petra2016,de2017sampling,haddock2019motzkin,haddock2019greed,rebrova2019sketching,morshed2020accelerated,morshed2020heavy} and references therein.

However, to the best of our knowledge,
there are few results in the literature that explore the use of greedy randomized Motzkin scheme, i.e., the  maximum residual rule, for Kaczmarz type algorithms for solving linear systems. 
To fill the research gap, in this work, paralleling to the GRK method, 
we develop the greedy randomized Kaczmarz method induced from the Motzkin method, i.e., the greedy randomized Motzkin-Kaczmarz (GRMK) method, for solving the systems (\ref{1}). Moreover, to further accelerate the GRMK method, we also present the block version of the new method using the 
index subset 
generated 
in the GRMK method 
and refer to it as the greedy Motzkin block Kaczmarz (GMBK) method.
Recently, many works on block Kaczmarz methods were reported because, compared with the original methods, the block methods allows for significant computational speedup and accelerated convergence to the solution;  see for example \cite{needell2014paved,needell2015randomized,gower2015randomized,briskman2015block}.
The block update formula can be written as
\begin{equation}
\label{rbgs}
x_{k+1}=x_{k}+A^{\dagger}_{\tau}(b_{\tau}-A_{\tau}x_{k}),
\end{equation}
where $\tau \subset \{1, \ldots, m\}$, $A_{\tau }$ and $b_{\tau}$ are the submatrix and subvector of $A$ and $b$, respectively, with rows indexed by $\tau$, and $A^{\dagger}_{\tau}$ is the Moore-Penrose pseudoinverse of $A_{\tau}$. 
To avoid computing the pseudoinverse, a variant of the above block Kaczmarz method is to project the current estimate onto each individual row that forms the submatrix $A_{\tau}$, and average the obtained projections 
 to form the next iterate:
\begin{equation}
\label{rbgss}
x_{k+1}=x_{k}-\sum_{i \in {\tau}} w_{i} \frac{A^{(i)} x_{k}-b^{(i)}}{\left\|A^{(i)}\right\|^{2}_2} (A^{(i)})^{T},
\end{equation}
 where $ w_{i}$ represents
the weight corresponding to the $i$-th row. 
This update is very suitable for 
distributed computing; see \cite{Necoara2019,Du20202,moorman2020randomized} for detailed discussions on this topic.

The rest of this paper is organized as follows. In Section \ref{sec2}, some notation and preliminaries are given. The GRMK method and its block version are discussed in Section \ref{sec3} and Section \ref{sec4}, respectively. Finally, we present the numerical results in Section \ref{sec5}.

\section{Notation and preliminaries}\label{sec2}
Throughout the paper, for a matrix $A$, 
${\rm R(A)}$ denotes its 
 column space, and for a set  $\mathcal{I}$, $|\mathcal{I}|$ denotes the number of elements of the set. 
In addition,  the smallest positive eigenvalues of $A^{T}A$ is denoted by $\lambda_{\min }(A^{T}A)$. 

To analyze the convergence of our new methods, the following fact will be used extensively.
\begin{lemma}
\label{lem2}\cite{Bai2018}
Let $A \in R^{m\times n}$ and for any vector $x \in {\rm R(A^T)}$, it holds that
\begin{align}
\|Ax\|^2_2\geq\lambda_{\min}\left( A^{T} A\right)\|x\|^2_2. \notag
\end{align}
\end{lemma}

For comparison later in this paper, we list the GRK method proposed in \cite{Bai2018} in Algorithm \ref{alg1}. 

\begin{alg}
\label{alg1}
The GRK method for the systems (\ref{1}).
\begin{enumerate}[]
\item \mbox{INPUT:} ~$A\in R^{m\times n}$, $b\in R^{m}$, $\ell$ , initial estimate $x_0$
\item \mbox{OUTPUT:} ~$x_\ell$
\item For $k=0, 1, 2, \ldots, \ell-1$ do
\item ~~~~Compute
\begin{align}
  \epsilon_{k}=\frac{1}{2}\left( \max _{1 \leq i_{k} \leq m}\left\{\frac{|r^{\left(i_{k}\right)}_k|^{2}}{\left\|A^{\left(i_{k}\right)}\right\|_{2}^{2}}\right\}+\frac{\left\|r_{k}\right\|_{2}^{2}}{\|A\|_{F}^{2}}\right).\notag
\end{align}
\item ~~~~Determine the index subset of positive integers
\begin{align}
 \mathcal{U}_{k}=\left\{i_{k}\Bigg|\frac{| r^{\left(i_{k}\right)}_k|^{2}} {\|A^{\left(i_{k}\right)}\|_{2}^{2}} \geq \epsilon_{k}\right\}.\notag
\end{align}

\item ~~~~Compute the $i$th entry $\tilde{r}_k^{(i)}$ of the vector $\tilde{r}_k$ according to
  $$
\tilde{r}_{k}^{(i)}=\left\{\begin{array}{ll}{r^{(i)}_{k},} & {\text { if } i \in \mathcal{U}_{k}}, \\ {0,} & {\text { otherwise. }}\end{array}\right.
$$
\item ~~~~Select $i_{k} \in \mathcal{U}_{k}$ with probability $\operatorname{Pr}\left(\mathrm{row}=i_{k}\right)=\frac{|\tilde{r}_{k}^{\left(i_{k}\right)}|^{2}}{\left\|\tilde{r}_{k}\right\|_{2}^{2}}$.
\item ~~~~Set
$$ x_{k+1}=x_{k}+\frac{   r^{ (i_{k}) }_{k} }{\| A^{(i_{k})}\|_{2}^{2}}( A^{(i_{k})})^{T}.$$
\item End for
\end{enumerate}
\end{alg}

The following greedy block Kaczmarz (GBK) method, i.e., Algorithm \ref{alg3}, was presented by Niu and Zheng \cite{Niu2020}, which can be seen as a block version of the GRK method. 
  \begin{alg}
\label{alg3}
The GBK method for the systems (\ref{1}).
\begin{enumerate}[]
\item \mbox{INPUT:} ~$A\in R^{m\times n}$, $b\in R^{m}$, $\ell$ , $\eta \in (0, 1]$, initial estimate $x_0$
\item \mbox{OUTPUT:} ~$x_\ell$
\item For $k=0, 1, 2, \ldots, \ell-1$ do
\item ~~~~Compute
\begin{align}
  \epsilon_{k}=\eta \cdot \max _{1 \leq i_{k} \leq m}\left\{\frac{|r^{\left(i_{k}\right)}_k|^{2}}{\left\|A^{\left(i_{k}\right)}\right\|_{2}^{2}}\right\}.\notag
\end{align}
\item ~~~~Determine the index subset of positive integers
\begin{align}
 \mathcal{U}_{k}=\left\{i_{k}\Bigg|\frac{| r^{\left(i_{k}\right)}_k|^{2}} {\|A^{\left(i_{k}\right)}\|_{2}^{2}} \geq \epsilon_{k}\right\}.\notag
\end{align}
\item ~~~~Set
\begin{align}
  x_{k+1}=x_{k}+A_{\mathcal{U}_{k}}^{\dagger}(b_{\mathcal{U}_{k}}-A_{\mathcal{U}_{k}}x_{k}).\notag
\end{align}
\item End for
\end{enumerate}
\end{alg}

 \section{The GRMK method }\label{sec3}
 The GRMK method is presented in Algorithm \ref{alg2}.
 Compared with the GRK method, the main differences are the methods for determining
the index subsets and the probability criterions for sampling an index. Specifically, the GRMK method 
determines the index subset $ \mathcal{I}_{k}$ using the 
combination of the maximum and average magnitude entries of the residual, and 
samples an index from the subset $ \mathcal{I}_{k}$ with probability that is 
proportional to the corresponding distance. On a high level, the GRMK method seems to change the order
of the first two main steps of Algorithm \ref{alg1}. However, it essentially comes from the maximum residual rule.

 \begin{alg}
\label{alg2}
The GRMK method for the systems (\ref{1}).
\begin{enumerate}[]
\item \mbox{INPUT:} ~$A\in R^{m\times n}$, $b\in R^{m}$, $\ell$ , initial estimate $x_0$
\item \mbox{OUTPUT:} ~$x_\ell$
\item For $k=0, 1, 2, \ldots, \ell-1$ do
\item ~~~~Compute
\begin{align}
  \delta_{k}=\frac{1}{2}\left( \max _{1 \leq i\leq m}|r^{\left(i\right)}_k|^{2}+\sum \limits _{i=1}^{m}\frac{ \|A^{(i)}\|_{2}^{2}}{\|A\|_{F}^{2}}|r^{\left(i\right)}_k|^{2}\right).\notag
\end{align}
\item ~~~~Determine the index subset of positive integers
\begin{align}
 \mathcal{I}_{k}=\left\{i_{k}\Bigg|| r^{\left(i_{k}\right)}_k|^{2}\geq \delta_{k}\right\}.\notag
\end{align}
\item ~~~~Compute the $i$th entry $\tilde{d}_k^{(i)}$ of the vector $\tilde{d}_k$ according to
  $$
\tilde{d}_{k}^{(i)}=\left\{\begin{array}{ll}{\frac{|r^{\left(i\right)}_k|^{2}}{\left\|A^{\left(i\right)}\right\|_{2}^{2}},} & {\text { if } i \in \mathcal{I}_{k}}, \\ {0,} & {\text { otherwise. }}\end{array}\right.
$$
\item ~~~~Select $i_{k} \in \mathcal{I}_{k}$ with probability $\operatorname{Pr}\left(\mathrm{row}=i_{k}\right)=\frac{ \tilde{d}_{k}^{\left(i_{k}\right)} }{\left\|\tilde{d}_{k}\right\|_{1}}$.
\item ~~~~Set
$$ x_{k+1}=x_{k}+\frac{   r^{ (i_{k}) }_{k} }{\| A^{(i_{k})}\|_{2}^{2}}( A^{(i_{k})})^{T}.$$
\item End for
\end{enumerate}
\end{alg}

\begin{remark}
\label{rmk1}
Note that if
$$|r^{ (i_{k}) }_{k}|^2 =  \max \limits _{1 \leq i  \leq m}|r^{ (i) }_{k}|^2 ,$$
then $i_k\in\mathcal{I}_{k}.$  This is because
$$
\max \limits _{1 \leq i  \leq m}|r^{ (i) }_{k}|^2 \geq \sum \limits _{i=1}^{m}\frac{ \|A^{(i)}\|_{2}^{2}}{\|A\|_{F}^{2}}|r^{\left(i\right)}_k|^{2}$$
and $$
|r^{ (i_{k}) }_{k}|^2 =  \max \limits _{1 \leq i  \leq m}|r^{ (i) }_{k}|^2 \geq\frac{1}{2}\left( \max _{1 \leq i \leq m}|r^{\left(i\right)}_k|^{2}+\sum \limits _{i=1}^{m}\frac{ \|A^{(i)}\|_{2}^{2}}{\|A\|_{F}^{2}}|r^{\left(i\right)}_k|^{2}\right).
$$
So the index subset $\mathcal{I}_{k}$ in Algorithm \ref{alg2} is always nonempty. 
\end{remark}

\begin{remark}
\label{rmk12}
As done in \cite{Bai2018r,Zhang2020}, we can introduce 
an arbitrary relaxation parameter $\theta \in [0, 1]$ into the
quantity $\delta_{k}$ in Algorithm \ref{alg2}, that is,
\begin{align}
  \delta_{k}= \theta\cdot \max _{1 \leq i\leq m}|r^{\left(i\right)}_k|^{2}+(1-\theta)\cdot\sum \limits _{i=1}^{m}\frac{ \|A^{(i)}\|_{2}^{2}}{\|A\|_{F}^{2}}|r^{\left(i\right)}_k|^{2}.\notag
\end{align}
Then, 
the relaxed greedy randomized Motzkin-Kaczmarz method can be devised. 
In this case, setting $\theta=1$, i.e., $\delta_{k}= \max \limits_{1 \leq i\leq m}|r^{\left(i\right)}_k|^{2}$, and $i_{k}={\rm arg} \max \limits _{i\in \mathcal{I}_{k}}\left\{\tilde{d}_{k}^{(i)}\right\}$, we can recover the 
greedy Kaczmarz method proposed in \cite{li2020novel}.
\end{remark}

Now, we bound the expected rate of convergence for Algorithm \ref{alg2}.
\begin{theorem}
\label{theorem1}
From an initial guess $x_0\in{\rm R(A^T)}$, the sequence $\{x_k\}_{k=0}^\infty$ generated by the GRMK method converges linearly in expectation to the least-Euclidean-norm solution $x_{\star}=A^{\dag}b$ and
\begin{align}
\textrm{E}\|x_{1}-x_\star\|^{2}_2
\leq\left(1- \frac{\min\limits_{i }\|A^{(i)}\|_{2}^{2}}{ \max \limits_{i \in \mathcal{I}_{k}}\|A^{(i)}\|^2_2}\frac{\lambda_{\min}(A^TA)}{\|A\|_{F}^{2}} \right)\| x_{0}-x_\star \|^{2}_2, \label{4eq}
\end{align}
and
\begin{eqnarray}
\textrm{E}\|x_{k+1}-x_\star\|^{2}_2\leq\left(1-\frac{1}{2}\frac{\min\limits_{i\neq i_{k-1}}\|A^{(i)}\|_{2}^{2}}{ \max\limits_{i \in \mathcal{I}_{k}}\|A^{(i)}\|^2_2}\frac{\lambda_{\min}(A^TA)}{\|A\|_{F}^{2}}\left(\frac{\|A\|_{F}^{2}}{\|A\|_{F}^{2}-\min \limits_{1 \leq i \leq m}\left\|A^{(i)}\right\|_{2}^{2}}+1\right)\right)\| x_{k}-x_\star \|^{2}_2,\notag
\\
\quad k=1, 2, \ldots.\label{5eq}
\end{eqnarray}
Moreover, let $\alpha=\min\left\{\frac{\min\limits_{i\neq i_{k-1}}\|A^{(i)}\|_{2}^{2}}{ \max\limits_{i \in \mathcal{I}_{k}}\|A^{(i)}\|^2_2}\right\}, k=1, 2, \ldots.$ Then
\begin{eqnarray}
\textrm{E}\|x_{k}-x_\star\|^{2}_2
&\leq& \left(1-\frac{\alpha}{2}\frac{\lambda_{\min}(A^TA)}{\|A\|_{F}^{2}}(\frac{\|A\|_{F}^{2}}{\|A\|_{F}^{2}-\min \limits_{1 \leq i \leq m}\left\|A^{(i)}\right\|_{2}^{2}}+1)\right)^{k-1}\notag
\\
&\times& \left(1- \frac{\min\limits_{i }\|A^{(i)}\|_{2}^{2}}{ \max \limits_{i \in \mathcal{I}_{k}}\|A^{(i)}\|^2_2}\frac{\lambda_{\min}(A^TA)}{\|A\|_{F}^{2}} \right) \| x_{0}-x_\star \|^{2}_2,
\quad k=1, 2, \ldots.\label{6eq}
\end{eqnarray}
\end{theorem}

\begin{proof}
From the update formula in Algorithm \ref{alg2}, we have
 $$x_{k+1}-x_{k}=\frac{ r^{(i_{k})}_k}{ \| A^{\left(i_{k}\right)} \|_{2}^{2}}( A^{(i_{k})})^{T},$$
which implies that $x_{k+1}-x_{k}$ is parallel to $( A^{(i_{k})})^{T}$. Meanwhile,
\begin{eqnarray*}
A^{\left(i_{k}\right)}(x_{k+1}-x_{\star})&=&A^{\left(i_{k}\right)}\left(x_{k}-x_\star+\frac{ r^{(i_{k})}_k}{ \| A^{\left(i_{k}\right)} \|_{2}^{2}}( A^{(i_{k})})^{T}\right)
\\
&=& A^{\left(i_{k}\right)}\left(x_{k}-x_\star\right)+r^{(i_{k})}_k,
\end{eqnarray*}
which together with the fact $Ax_{\star}=b$ gives
\begin{eqnarray*}
A^{\left(i_{k}\right)}(x_{k+1}-x_{\star})&=& ( A^{\left(i_{k}\right)}x_{k}-b^{\left(i_{k}\right)})+(b^{\left(i_{k}\right)}- A^{\left(i_{k}\right)}x_{k})=0.
\end{eqnarray*}
Then $x_{k+1}-x_{\star}$ is orthogonal to $A^{\left(i_{k}\right)}$. Thus, the vector $x_{k+1}-x_{k}$ is perpendicular to the vector $x_{k+1}-x_{\star}$. By the Pythagorean theorem, we get
\begin{equation*}
\label{7}
\|x_{k+1}-x_{\star}\|^2_2=\|x_{k}-x_{\star}\|^2_2-\|x_{k+1}-x_{k}\|^2_2.
\end{equation*}
Now, taking expectation of both sides, we have
\begin{align}
\textrm{E}\|x_{k+1}-x_\star\|^{2}_2
&=\| x_{k}-x_\star \|^{2}_2- \textrm{E}\|x_{k+1}-x_{k}\|^2_2 \notag
\\
&= \| x_{k}-x_\star \|^{2}_2- \sum  \limits _{i_k \in \mathcal{I}_{k}}\frac{\tilde{d}_{k}^{\left(i_{k}\right)}}{\sum  \limits _{i_k \in \mathcal{I}_{k}}\tilde{d}_{k}^{\left(i_{k}\right)}} \frac{ |r^{(i_{k})}_k|^2}{ \| A^{\left(i_{k}\right)} \|_{2}^{2}}\notag
\\
&\leq\| x_{k}-x_\star \|^{2}_2- \frac{1}{ \max\limits_{i \in \mathcal{I}_{k}}\|A^{(i)}\|^2_2}\sum  \limits _{i_k \in \mathcal{I}_{k}}\frac{\tilde{d}_{k}^{\left(i_{k}\right)}}{\sum  \limits _{i_k \in \mathcal{I}_{k}}\tilde{d}_{k}^{\left(i_{k}\right)}} |r^{(i_{k})}_k|^2.\label{eq1}
\end{align}

For $k=0$, according to Algorithm \ref{alg2}, we have
\begin{align}
| r^{\left(i_{0}\right)}_0|^{2}
&\geq \frac{1}{2}\left( \max _{1 \leq i\leq m}|r^{\left(i\right)}_0|^{2}+\sum \limits _{i=1}^{m}\frac{ \|A^{(i)}\|_{2}^{2}}{\|A\|_{F}^{2}}|r^{\left(i\right)}_0|^{2}\right)\notag
\\
& = \frac{1}{2}\sum \limits _{i=1 }^{m}\frac{ \|A^{(i)}\|_{2}^{2}}{\|A\|_{F}^{2}}|r^{\left(i\right)}_0|^{2}(\frac{\max \limits_{1 \leq i\leq m}|r^{\left(i\right)}_0|^{2}}{\sum \limits _{i=1 }^{m}\frac{ \|A^{(i)}\|_{2}^{2}}{\|A\|_{F}^{2}}|r^{\left(i\right)}_0|^{2}}+1)\notag
\\
&\geq \frac{1}{2}\sum \limits_{i=1 }^{m}\frac{ \|A^{(i)}\|_{2}^{2}}{\|A\|_{F}^{2}}|r^{\left(i\right)}_0|^{2}(\frac{1}{\sum \limits _{i=1 }^{m}\frac{ \|A^{(i)}\|_{2}^{2}}{\|A\|_{F}^{2}}}+1)\notag
\\
&\geq \frac{1}{2}\frac{\min\limits_{i}\|A^{(i)}\|_{2}^{2}}{\|A\|_{F}^{2}} \sum \limits _{i=1 }^{m}|r^{\left(i\right)}_0|^{2}(\frac{1}{\sum \limits _{i=1 }^{m}\frac{ \|A^{(i)}\|_{2}^{2}}{\|A\|_{F}^{2}}}+1)\notag
\\
&=  \frac{\min\limits_{i }\|A^{(i)}\|_{2}^{2}}{\|A\|_{F}^{2}} \|r_0\|_2^{2},\notag
\end{align}
which together with Lemma \ref{lem2} yields
\begin{align}
| r^{\left(i_{0}\right)}_0|^{2}
&\geq \frac{\min\limits_{i }\|A^{(i)}\|_{2}^{2}\lambda_{\min}(A^TA)}{\|A\|_{F}^{2}} \| x_{0}-x_\star \|^{2}_2.\label{eq21}
\end{align}
Thus, substituting (\ref{eq21}) into (\ref{eq1}), we get
\begin{align}
\textrm{E}\|x_{1}-x_\star\|^{2}_2
&\leq\| x_{0}-x_\star \|^{2}_2-\frac{\min\limits_{i }\|A^{(i)}\|_{2}^{2}\lambda_{\min}(A^TA)}{\max \limits_{i \in \mathcal{I}_{k}}\|A^{(i)}\|^2_2\|A\|_{F}^{2}} \| x_{0}-x_\star \|^{2}_2 \notag
\\
&\leq\left(1- \frac{\min\limits_{i }\|A^{(i)}\|_{2}^{2}}{ \max \limits_{i \in \mathcal{I}_{k}}\|A^{(i)}\|^2_2}\frac{\lambda_{\min}(A^TA)}{\|A\|_{F}^{2}} \right)\| x_{0}-x_\star \|^{2}_2,  \notag
\end{align}
which is just the estimate (\ref{4eq}).

For $k\geq1$, to find the lower bound of $| r^{\left(i_{k}\right)}_k|^{2}$, first note that
\begin{align} r_{k}^{\left(i_{k-1}\right)} &=b^{\left(i_{k-1}\right)}-A^{\left(i_{k-1}\right)} x_{k} \notag\\ &=b^{\left(i_{k-1}\right)}-A^{\left(i_{k-1}\right)}\left(x_{k-1}+\frac{r^{\left(i_{k-1}\right)}_{k-1}}{\left\|A^{\left(i_{k-1}\right)}\right\|_{2}^{2}}\left(A^{\left(i_{k-1}\right)}\right)^{T}\right) \notag
\\ &=b^{\left(i_{k-1}\right)}-A^{\left(i_{k-1}\right)} x_{k-1}-r^{\left(i_{k-1}\right)}_{k-1} \notag
\\
 &=0,\label{eq0}
 \end{align}
and
\begin{align}
| r^{\left(i_{k}\right)}_k|^{2}
&\geq \frac{1}{2}\left( \max _{1 \leq i\leq m}|r^{\left(i\right)}_k|^{2}+\sum \limits _{i=1}^{m}\frac{ \|A^{(i)}\|_{2}^{2}}{\|A\|_{F}^{2}}|r^{\left(i\right)}_k|^{2}\right)\notag
\\
& = \frac{1}{2}\sum \limits _{i=1,i\neq i_{k-1}}^{m}\frac{ \|A^{(i)}\|_{2}^{2}}{\|A\|_{F}^{2}}|r^{\left(i\right)}_k|^{2}(\frac{\max \limits_{1 \leq i\leq m}|r^{\left(i\right)}_k|^{2}}{\sum \limits _{i=1,i\neq i_{k-1}}^{m}\frac{ \|A^{(i)}\|_{2}^{2}}{\|A\|_{F}^{2}}|r^{\left(i\right)}_k|^{2}}+1)\notag
\\
&\geq \frac{1}{2}\sum \limits _{i=1,i\neq i_{k-1}}^{m}\frac{ \|A^{(i)}\|_{2}^{2}}{\|A\|_{F}^{2}}|r^{\left(i\right)}_k|^{2}(\frac{1}{\sum \limits _{i=1,i\neq i_{k-1}}^{m}\frac{ \|A^{(i)}\|_{2}^{2}}{\|A\|_{F}^{2}}}+1)\notag
\\
&\geq \frac{1}{2}\frac{\min\limits_{i\neq i_{k-1}}\|A^{(i)}\|_{2}^{2}}{\|A\|_{F}^{2}} \sum \limits _{i=1,i\neq i_{k-1}}^{m}|r^{\left(i\right)}_k|^{2}(\frac{1}{\sum \limits _{i=1,i\neq i_{k-1}}^{m}\frac{ \|A^{(i)}\|_{2}^{2}}{\|A\|_{F}^{2}}}+1)\notag
\\
&= \frac{1}{2}\frac{\min\limits_{i\neq i_{k-1}}\|A^{(i)}\|_{2}^{2}}{\|A\|_{F}^{2}}(\frac{1}{\sum \limits _{i=1,i\neq i_{k-1}}^{m}\frac{ \|A^{(i)}\|_{2}^{2}}{\|A\|_{F}^{2}}}+1)\|r_k\|_2^{2}.\notag
\end{align}
Further, considering Lemma \ref{lem2}, we have
\begin{align}
| r^{\left(i_{k}\right)}_k|^{2}
&\geq \frac{1}{2}\frac{\min\limits_{i\neq i_{k-1}}\|A^{(i)}\|_{2}^{2}}{\|A\|_{F}^{2}}(\frac{1}{\sum \limits _{i=1,i\neq i_{k-1}}^{m}\frac{ \|A^{(i)}\|_{2}^{2}}{\|A\|_{F}^{2}}}+1)\lambda_{\min}(A^TA)\| x_{k}-x_\star \|^{2}_2\notag
\\
&\geq\frac{1}{2}\frac{\min\limits_{i\neq i_{k-1}}\|A^{(i)}\|_{2}^{2}\lambda_{\min}(A^TA)}{\|A\|_{F}^{2}}(\frac{\|A\|_{F}^{2}}{\|A\|_{F}^{2}-\min \limits_{1 \leq i \leq m}\left\|A^{(i)}\right\|_{2}^{2}}+1)\| x_{k}-x_\star \|^{2}_2.\label{eq2}
\end{align}
Thus, substituting (\ref{eq2}) into (\ref{eq1}), we get the estimate (\ref{5eq}). By induction on the iteration index $k$, we can obtain the estimate (\ref{6eq}).
\end{proof}
\begin{remark}
\label{rmk2}
According to 
(\ref{eq0}), we know that $r_{k}^{\left(i_{k-1}\right)}=0$, which implies that $i_{k-1}\notin \mathcal{I}_{k}$. So the GRMK method can make sure 
the same index will never be chosen twice in iteration and we also have
\begin{align}
\frac{\min\limits_{i\neq i_{k-1}}\|A^{(i)}\|_{2}^{2}}{ \max\limits_{i \in \mathcal{I}_{k}}\|A^{(i)}\|^2_2}\leq1.\label{223w0}
\end{align}
\end{remark}
\begin{remark}
\label{rmk3}
For the GRK method, the error estimate in expectation given in \cite{Bai2018} is
\begin{align}
\textrm{E}\left\|x_{k+1}-x_{\star}\right\|_{2}^{2} \leq\left(1-\frac{1}{2} \frac{\lambda_{\min}(A^TA)}{\|A\|_{F}^{2}}\left(\frac{\|A\|_{F}^{2}}{\|A\|_{F}^{2}-\min \limits_{1 \leq i \leq m}\left\|A^{(i)}\right\|_{2}^{2}}+1\right)\right)\| x_{k}-x_\star \|^{2}_2. \label{223wewe23000}
\end{align}
Combining (\ref{223w0}) and (\ref{223wewe23000}), we can get
\begin{align}
&1-\frac{1}{2} \frac{\lambda_{\min}(A^TA)}{\|A\|_{F}^{2}}\left(\frac{\|A\|_{F}^{2}}{\|A\|_{F}^{2}-\min \limits_{1 \leq i \leq m}\left\|A^{(i)}\right\|_{2}^{2}}+1\right)\notag
\\
&\leq 1-\frac{1}{2}\frac{\min\limits_{i\neq i_{k-1}}\|A^{(i)}\|_{2}^{2}}{ \max\limits_{i \in \mathcal{I}_{k}}\|A^{(i)}\|^2_2}\frac{\lambda_{\min}(A^TA)}{\|A\|_{F}^{2}}\left(\frac{\|A\|_{F}^{2}}{\|A\|_{F}^{2}-\min \limits_{1 \leq i \leq m}\left\|A^{(i)}\right\|_{2}^{2}}+1\right)<1.\notag
\end{align}
That is, the convergence factor of the GRMK method is indeed smaller than 1 and is larger 
than that of the GRK method. However, as pointed out in \cite{Bai2018r}, the convergence factor only describes the worst case of the algorithm and is just the upper bound of 
the actual convergence rate. 
So, these convergence factors can not be used to evaluate the actual convergence speed of algorithms directly. To make this fact clearer, we present some numerical results in Fig. \ref{fi1} to illustrate 
the convergence factors and the actual convergence rates of the GRMK and GRK methods, where 
the definition of the actual convergence rate is taken from \cite{Bai2018r}
\begin{align}\label{2convergence}
\rho_{k} =\left(\frac{\textrm{E}\left\|x_{k}-x_{\star}\right\|_{2}^{2}}{\left\|x_{0}-x_{\star}\right\|_{2}^{2}}\right)^{1 / k}, \quad k \geq 1.
\end{align}

\begin{figure}[!htb]
 \begin{center}
\includegraphics [height=4cm,width=5.5cm ]{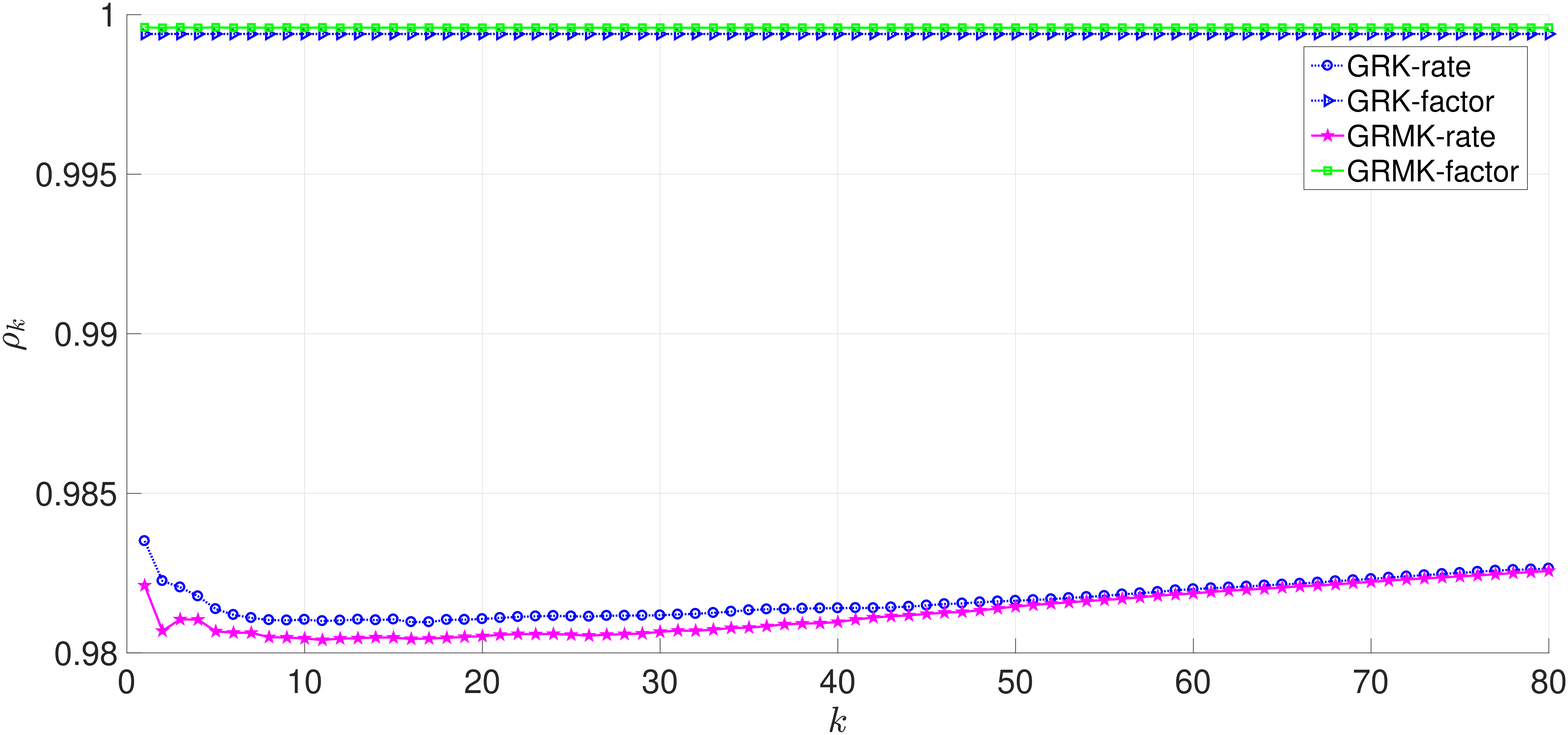}
\includegraphics [height=4cm,width=5.5cm ]{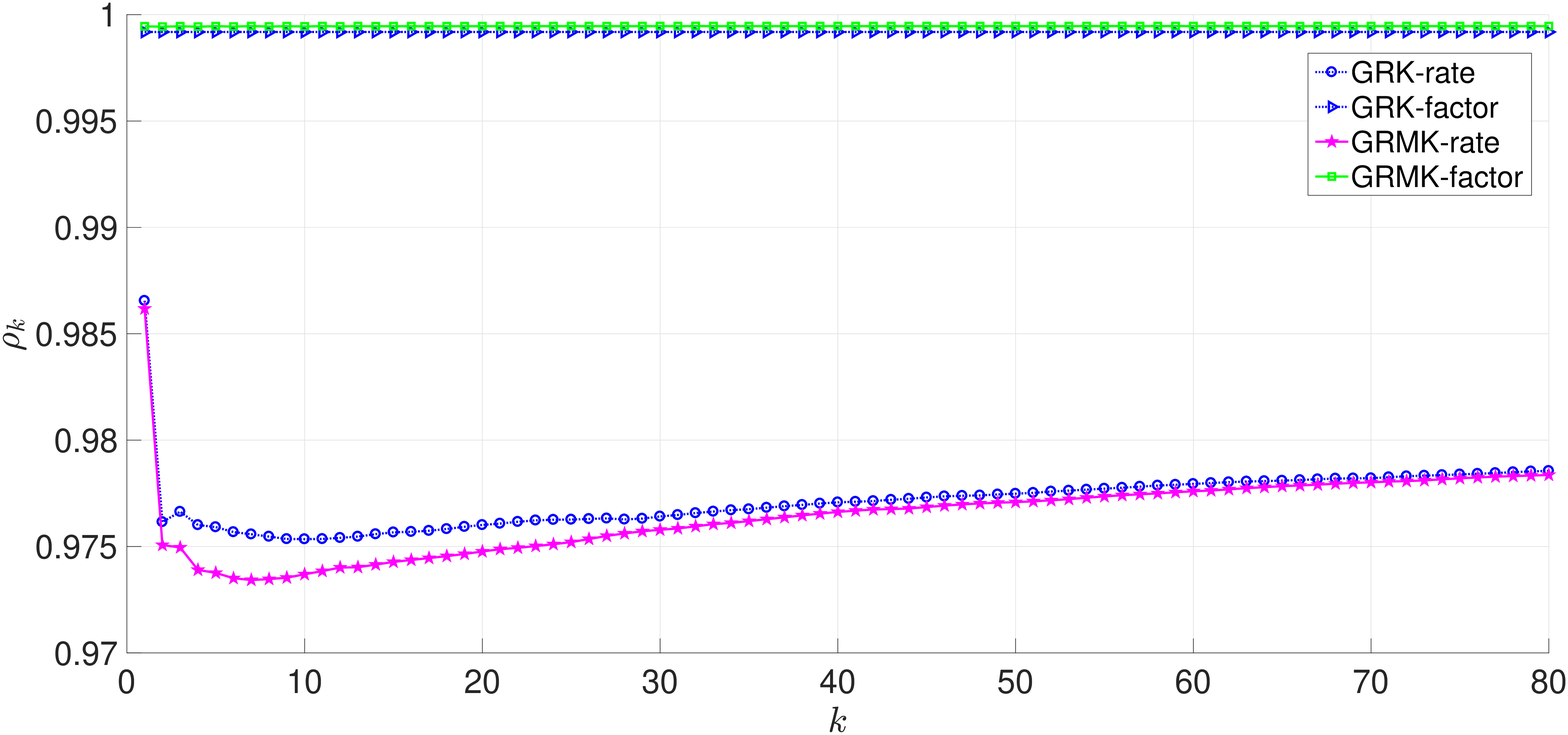}
 \end{center}
\caption{\rm Convergence factors/rates of the GRK and GRMK methods with matrices whose entries are uniformly distributed random numbers between 0 and 1. (left) $A$ is of order $2000\times 200$; (right) $A$ is of order $5000\times 200$. GRK (GRMK)-factor=convergence factor of GRK (GRMK) method;  GRK (GRMK)-rate=actual convergence rate of GRK (GRMK) method. }\label{fi1}
\end{figure}

Numerical results show that 
the convergence factors of the GRMK are indeed a little larger than those of the GRK method. However, the actual convergence rates of the GRMK method are a little smaller than those of the GRK method. In addition, we can also find that the convergence factors are the quite loose upper bounds of the actual convergence rates. 
\end{remark}

 \section{The GMBK method}\label{sec4}
The GMBK method is presented in Algorithm \ref{alg4}. 
Unlike the GRMK method, after determining the index subset $\mathcal{I}_{k}$, the GMBK method projects the current iterate onto the solution space of this subset simultaneously.
 \begin{alg}
\label{alg4}
The GMBK method for the systems (\ref{1}).
\begin{enumerate}[]
\item \mbox{INPUT:} ~$A\in R^{m\times n}$, $b\in R^{m}$, $\ell$ , initial estimate $x_0$
\item \mbox{OUTPUT:} ~$x_\ell$
\item For $k=0, 1, 2, \ldots, \ell-1$ do
\item ~~~~Compute
\begin{align}
  \delta_{k}=\frac{1}{2}\left( \max _{1 \leq i\leq m}|r^{\left(i\right)}_k|^{2}+\sum \limits _{i=1}^{m}\frac{ \|A^{(i)}\|_{2}^{2}}{\|A\|_{F}^{2}}|r^{\left(i\right)}_k|^{2}\right).\notag
\end{align}
\item ~~~~Determine the index subset of positive integers
\begin{align}
 \mathcal{I}_{k}=\left\{i_{k}\Bigg|| r^{\left(i_{k}\right)}_k|^{2}\geq \delta_{k}\right\}.\notag
\end{align}
\item ~~~~Set
\begin{align}
  x_{k+1}=x_{k}+A_{\mathcal{I}_{k}}^{\dagger}(b_{\mathcal{I}_{k}}-A_{\mathcal{I}_{k}}x_{k}).\notag
\end{align}
\item End for
\end{enumerate}
\end{alg}
\begin{remark}
\label{rmk4134}
In the GMBK method, we adopt the iterative format (\ref{rbgs}) to update the approximation. We can also update it in the form of the formula (\ref{rbgss}) to avoid computing the pseudoinverse.
\end{remark}

\begin{remark}
\label{rmk413344}
 Similar to Algorithm \ref{alg3}, i.e., the GBK method, we can set the quantity $ \delta_{k}$ in Algorithm \ref{alg4} in the following form 
 \begin{align}
  \delta_{k}=\xi \max _{1 \leq i\leq m}|r^{\left(i\right)}_k|^{2},\notag
\end{align}
where $\xi\in (0, 1]$ is a parameter.
\end{remark}

Next, we bound the rate of convergence for Algorithm \ref{alg4}.
\begin{theorem}
\label{theorem2}
From an initial guess $x_0\in{\rm R(A^T)}$, the sequence $\{x_k\}_{k=0}^\infty$ generated by the GMBK method converges linearly to the least-Euclidean-norm solution $x_{\star}=A^{\dag}b$ and
\begin{align}
\|x_{1}-x_\star\|^{2}_2
 \leq (1-|\mathcal{I}_{0}|\frac{\min\limits_{i }\|A^{(i)}\|_{2}^{2}}{\lambda_{\max} (A_{\mathcal{I}_{0}}^TA_{\mathcal{I}_{0}})}\frac{\lambda_{\min}(A^TA)}{\|A\|_{F}^{2}} )\| x_{0}-x_\star \|^{2}_2, \label{1a}
\end{align}
and
\begin{align}
 \|x_{k+1}-x_\star\|^{2}_2\leq(1-\frac{1}{2}|\mathcal{I}_{k}|\frac{\min\limits_{i \notin \mathcal{I}_{k-1}}\|A^{(i)}\|_{2}^{2}}{\lambda_{\max}    (A_{\mathcal{I}_{k}}^TA_{\mathcal{I}_{k}})}\frac{\lambda_{\min}(A^TA)}{\|A\|_{F}^{2}}(\frac{\|A\|_{F}^{2}}{{\|A\|_{F}^{2}-\left\|A_{\mathcal{I}_{k-1}}\right\|_{F}^{2}}}+1))\| x_{k}-x_\star \|^{2}_2,\notag
\\
\quad k=1, 2, \ldots. \label{2a}
\end{align}
\end{theorem}

\begin{proof}
From Algorithm \ref{alg4}, using the fact $Ax_\star=b$, we have
\begin{eqnarray*}
x_{k+1}-x_\star&=&x_{k}-x_\star+A_{\mathcal{I}_{k}}^{\dagger}(b_{\mathcal{I}_{k}}-A_{\mathcal{I}_{k}}x_{k})
\\
&=& x_{k}-x_\star-A_{\mathcal{I}_{k}}^{\dagger}A_{\mathcal{I}_{k}}(x_{k}-x_{\star})
\\
&=&  (I- A_{\mathcal{I}_{k}}^{\dagger}A_{\mathcal{I}_{k}} )( x_{k}-x_\star).
\end{eqnarray*}
Since $A_{\mathcal{I}_{k}}^{\dagger}A_{\mathcal{I}_{k}}$ is an orthogonal projector, taking the square of the Euclidean norm on both sides and applying Pythagorean theorem, we get
\begin{align}
\|x_{k+1}-x_\star\|^{2}_2
&=   \|(I- A_{\mathcal{I}_{k}}^{\dagger}A_{\mathcal{I}_{k}} )( x_{k}-x_\star)\|^{2}_2 \notag
\\
&=  \|   x_{k}-x_\star \|^{2}_2- \| A_{\mathcal{I}_{k}}^{\dagger}A_{\mathcal{I}_{k}} ( x_{k}-x_\star)\|^{2}_2,  \notag
\end{align}
which together with Lemma \ref{lem2} and the fact $\lambda_{\min} ((A_{\mathcal{I}_{k}}^{\dagger})^TA_{\mathcal{I}_{k}}^{\dagger})=\lambda_{\max}^{-1}(A_{\mathcal{I}_{k}}^TA_{\mathcal{I}_{k}})$ yields
\begin{align}
\|x_{k+1}-x_\star\|^{2}_2
&\leq  \|   x_{k}-x_\star \|^{2}_2-\lambda_{\min} ((A_{\mathcal{I}_{k}}^{\dagger})^TA_{\mathcal{I}_{k}}^{\dagger})\| A_{\mathcal{I}_{k}} ( x_{k}-x_\star)\|^{2}_2\notag
\\
&= \| x_{k}-x_\star \|^{2}_2-\lambda_{\max}^{-1}(A_{\mathcal{I}_{k}}^TA_{\mathcal{I}_{k}})\| A_{\mathcal{I}_{k}} ( x_{k}-x_\star)\|^{2}_2 \notag
\\
&= \| x_{k}-x_\star \|^{2}_2-\lambda_{\max}^{-1}(A_{\mathcal{I}_{k}}^TA_{\mathcal{I}_{k}})\sum \limits _{i_k\in \mathcal{I}_{k}}| r^{\left(i_{k}\right)}_k|^{2}.   \label{er}
\end{align}
Now, the main task is to find the lower bound of $| r^{\left(i_{k}\right)}_k|^{2}$.

For $k=0$, substituting (\ref{eq21}) into (\ref{er}), we quickly obtain
\begin{align}
\|x_{1}-x_\star\|^{2}_2
&\leq \| x_{0}-x_\star \|^{2}_2-\lambda_{\max}^{-1}(A_{\mathcal{I}_{0}}^TA_{\mathcal{I}_{0}})|\mathcal{I}_{0}|\frac{\min\limits_{i }\|A^{(i)}\|_{2}^{2}\lambda_{\min}(A^TA)}{\|A\|_{F}^{2}} \| x_{0}-x_\star \|^{2}_2 \notag
\\
&\leq (1-|\mathcal{I}_{0}|\frac{\min\limits_{i }\|A^{(i)}\|_{2}^{2}}{\lambda_{\max} (A_{\mathcal{I}_{0}}^TA_{\mathcal{I}_{0}})}\frac{\lambda_{\min}(A^TA)}{\|A\|_{F}^{2}} )\| x_{0}-x_\star \|^{2}_2, \notag
\end{align}
which is just the estimate (\ref{1a}).

For $k\geq1$, similar to the derivation of the inequality (\ref{eq2}), we obtain
\begin{align}
| r^{\left(i_{k}\right)}_k|^{2}
&\geq \frac{1}{2}\frac{\min\limits_{i \notin \mathcal{I}_{k-1}}\|A^{(i)}\|_{2}^{2}\lambda_{\min}(A^TA)}{\|A\|_{F}^{2}}(\frac{\|A\|_{F}^{2}}{\sum \limits _{i=1,i\notin \mathcal{I}_{k-1}}^{m}\|A^{(i)}\|_{2}^{2}}+1)\| x_{k}-x_\star \|^{2}_2 \notag
\\
&= \frac{1}{2}\frac{\min\limits_{i \notin \mathcal{I}_{k-1}}\|A^{(i)}\|_{2}^{2}\lambda_{\min}(A^TA)}{\|A\|_{F}^{2}}(\frac{\|A\|_{F}^{2}}{\|A\|_{F}^{2}-\left\|A_{\mathcal{I}_{k-1}}\right\|_{F}^{2}}+1)\| x_{k}-x_\star \|^{2}_2. \label{eq21212}
\end{align}
Then, substituting (\ref{eq21212}) into (\ref{er}), we get
\begin{align}
\|x_{k+1}-x_\star\|^{2}_2
&\leq   \| x_{k}-x_\star \|^{2}_2-\frac{1}{2}|\mathcal{I}_{k}|\frac{\min\limits_{i \notin \mathcal{I}_{k-1}  }\|A^{(i)}\|_{2}^{2}}{\lambda_{\max}    (A_{\mathcal{I}_{k}}^TA_{\mathcal{I}_{k}})}\frac{\lambda_{\min}(A^TA)}{\|A\|_{F}^{2}}(\frac{\|A\|_{F}^{2}}{\|A\|_{F}^{2}-\left\|A_{\mathcal{I}_{k-1}}\right\|_{F}^{2}}+1)\| x_{k}-x_\star \|^{2}_2 \notag
\\
&=(1-\frac{1}{2}|\mathcal{I}_{k}|\frac{\min\limits_{i \notin \mathcal{I}_{k-1}}\|A^{(i)}\|_{2}^{2}}{\lambda_{\max}    (A_{\mathcal{I}_{k}}^TA_{\mathcal{I}_{k}})}\frac{\lambda_{\min}(A^TA)}{\|A\|_{F}^{2}}(\frac{\|A\|_{F}^{2}}{\|A\|_{F}^{2}-\left\|A_{\mathcal{I}_{k-1}}\right\|_{F}^{2}}+1)\| x_{k}-x_\star \|^{2}_2,\notag
\end{align}
which implies the desired result (\ref{2a}).
\end{proof}

\begin{remark}
\label{rmk5}
From Algorithms \ref{alg2} and \ref{alg4}, we know that $\{x|A_{\mathcal{I}_{k}}x=b_{\mathcal{I}_{k}}\}\subset\{x|A^{(i_{k})}x=b^{(i_{k})}\}$, where $i_{k}$ is the update index of the GRMK method. Thus, similar to the analysis in \cite{haddock2019greed}, we can obtain $$\|x_k^{GRMK}-x_{k-1}\|^2_2\leq\|x_k^{GMBK}-x_{k-1}\|^2_2,$$
which together with the fact
\begin{align*}
\left\|{x}_{k}^{GRMK}-{x}_{k-1}\right\|^{2}_2+\left\|{x}_{k}^{GRMK}-{x}_{\star}\right\|^{2}_2&=\left\|{x}_{k-1}-{x}_{\star}\right\|^{2}_2\\
&=\left\|{x}_{k}^{GMBK}-{x}_{k-1}\right\|^{2}_2+\left\|{x}_{k}^{GMBK}-{x}_{\star}\right\|^{2}_2, \notag
\end{align*}
leads to
$$\|x_k^{GMBK}-x_{\star}\|^2_2\leq\|x_k^{GRMK}-x_{\star}\|^2_2.$$
In the above expressions, $x_k^{GMBK}$ and $x_k^{GRMK}$ denote the next approximations generated by the GMBK and GRMK methods, respectively. Hence, the GMBK method converges at least as fast as the GRMK method.

In addition, since $\{x|A_{\mathcal{I}_{k}}x=b_{\mathcal{I}_{k}}\}\subset\{x|A^{(l_{k})}x=b^{(l_{k})}\}$, where $l_k={\rm arg} \max \limits _{1\leq i \leq m} |r^{\left(i\right)}_k|^{2} $, we can also get that the GMBK method must converge at least as fast as the Motzkin method.
\end{remark}

\begin{remark}
\label{rmk7}
To compare Algorithms \ref{alg3} and \ref{alg4} fairly, in the following, we set $\eta$ in Algorithm \ref{alg3} to be
 $$\eta=\frac{1}{2}+\frac{1}{2} \frac{\left\|b-A x_{k}\right\|_{2}^{2}}{\|A\|_{F}^{2}}\left(\max _{1 \leq i_k \leq m}\left\{\frac{\left|b^{i_k}-A^{( i_k)} x_{k}\right|^{2}}{\left\|A^{(i_k)}\right\|_{2}^{2}}\right\}\right)^{-1},$$
that is, set
\begin{align}
  \epsilon_{k}=\frac{1}{2}\left( \max _{1 \leq i_{k} \leq m}\left\{\frac{|r^{\left(i_{k}\right)}_k|^{2}}{\left\|A^{\left(i_{k}\right)}\right\|_{2}^{2}}\right\}+\frac{\left\|r_{k}\right\|_{2}^{2}}{\|A\|_{F}^{2}}\right). \notag
\end{align}
We refer to this block algorithm as the greedy distance block Kaczmarz (GDBK) method.

From Remark \ref{rmk3}, we know that the convergence factor cannot accurately explain the convergence speed of a method. So, we compare the actual convergence rates defined in (\ref{2convergence}) of the GDBK and GMBK methods using numerical experiments. The numerical results are listed in Fig. \ref{fi2}, which 
show that these actual convergence rates are almost the same.

\begin{figure}[!htb]
 \begin{center}
\includegraphics [height=4cm,width=5.5cm ]{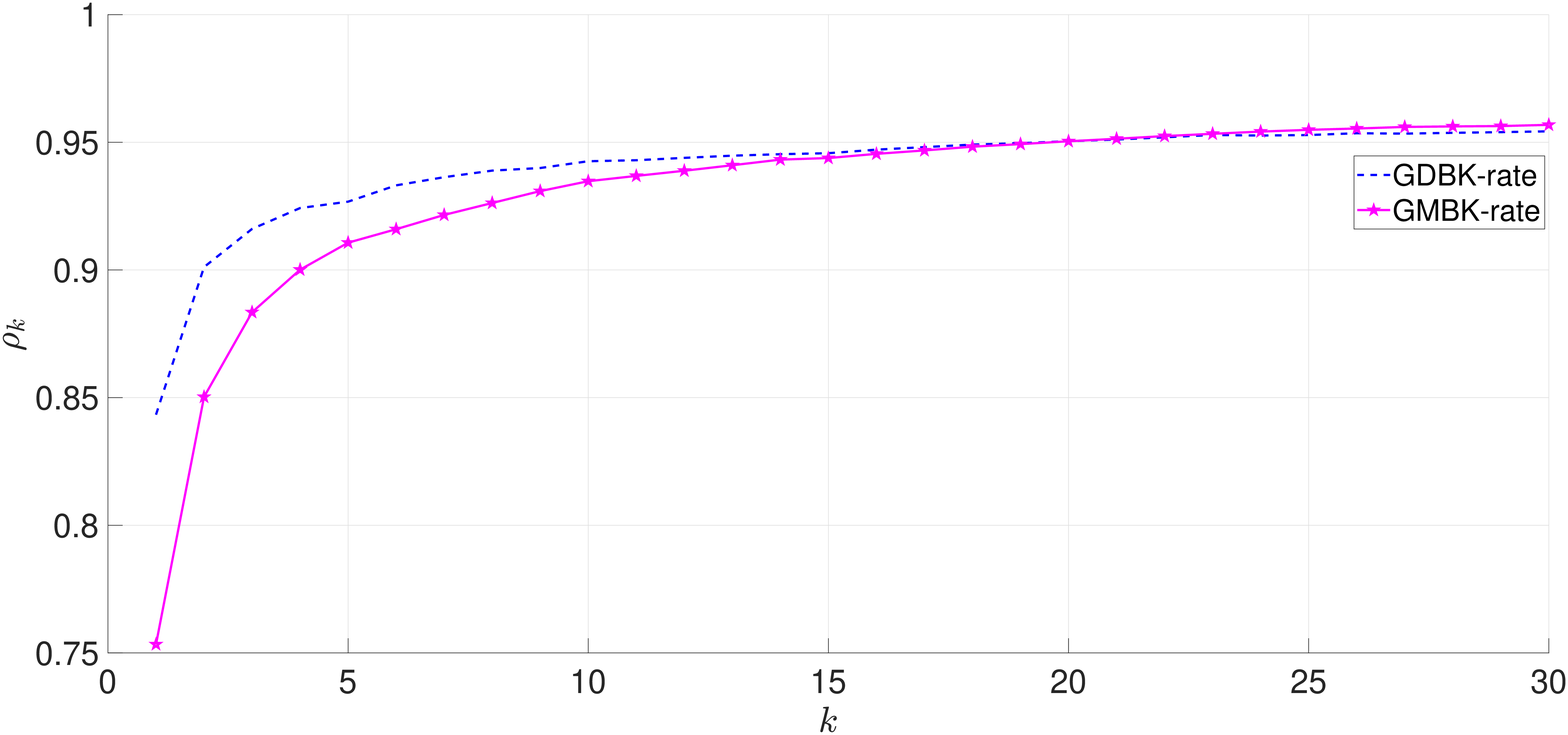}
\includegraphics [height=4cm,width=5.5cm ]{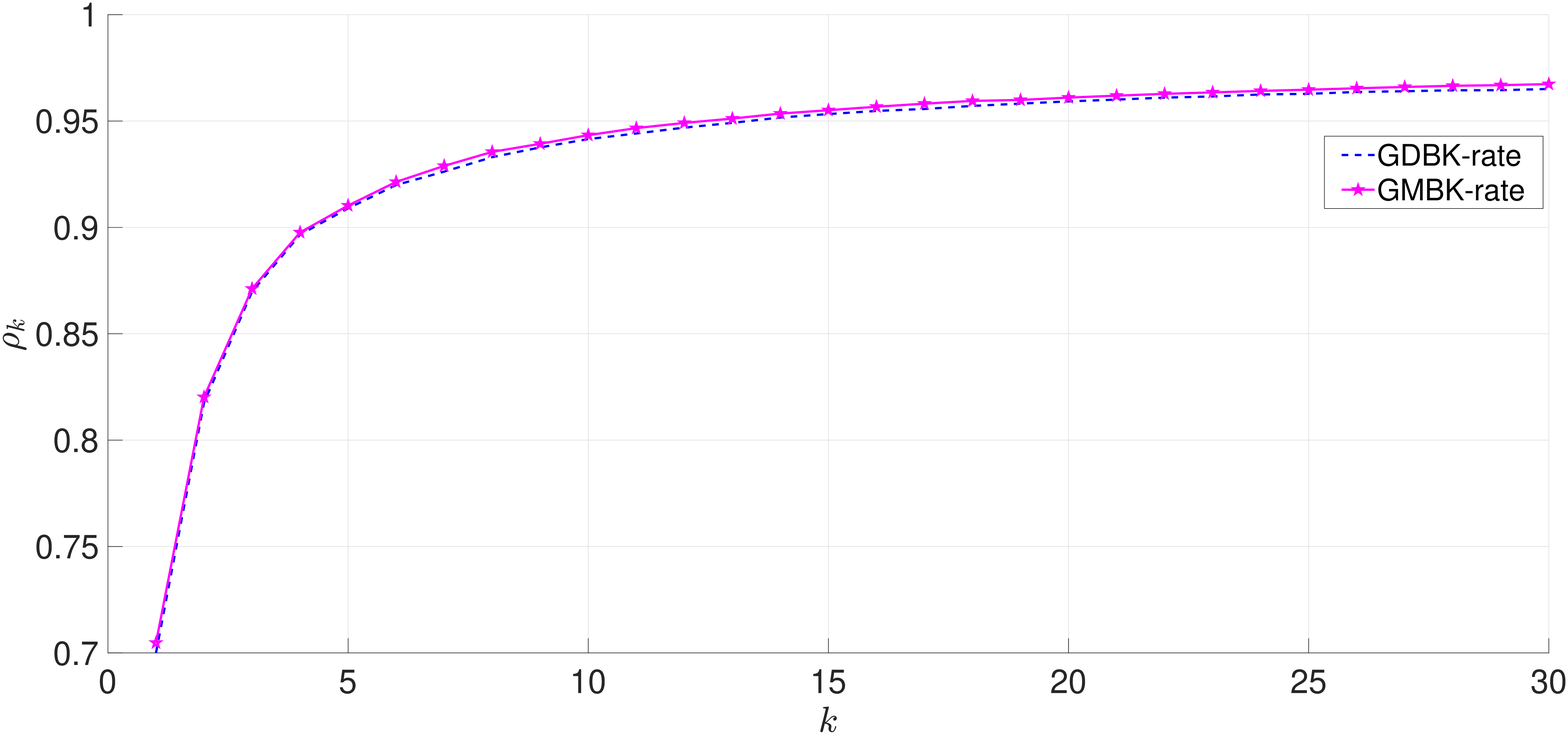}
 \end{center}
\caption{\rm Actual convergence rates of the GDBK and GMBK methods with matrices whose entries are uniformly distributed random numbers between 0 and 1. (left) $A$ is of order $1000\times 200$; (right) $A$ is of order $5000\times 1000$. GDBK (GMBK)-rate=actual convergence rate of GDBK (GMBK) method. }\label{fi2}
\end{figure}

\end{remark}

 \section{Numerical experiments}\label{sec5}
In this section, we mainly compare our new greedy Motzkin-Kaczmarz methods (GRMK, GMBK) with the greedy distance Kaczmarz methods (GRK, GDBK) in terms of the iteration numbers (denoted as ``Iteration'') and computing time in seconds (denoted as ``CPU time(s)'') with different matrices $A\in R^{m\times n}$. In all the following specific experiments, we generate the solution vector $x_\star\in R^{n}$ using the MATLAB function \texttt{randn}, and the vector $b\in R^{m}$ by setting $b=Ax_\star$. All experiments start from an initial vector $x_0=0$, and terminate once the \emph{relative solution error }(RES) or \emph{relative residual} (RR) at $x_{k}$ is less than $10^{-10}$, where RES and RR are defined by $$\rm RES=\frac{\left\|x_{k}-A^{\dagger}b\right\|^{2}_2}{\left\|A^{\dagger}b\right\|^{2}_2},~~\mathrm{RR}=\frac{\left\|b-A x_{k}\right\|_2^2}{\left\|b-A x_{0}\right\|_2^2}.$$

We first consider three main models of the coefficient matrix $A$: a Gaussian matrix with i.i.d. $N(0, 1)$ entries generated by the MATLAB function \texttt{randn}, a sparse normally distributed random matrix generated by the  MATLAB function \texttt{sprandn(m,n,0.2,0.8)}, and a sparse uniformly distributed random matrix generated by the  MATLAB function \texttt{sprand(m,n,0.2,0.8)}. Numerical results are reported in Figures \ref{fig1}--\ref{fig6}, which describe the $\log_{10}$(RES) or $\log_{10}$(RR) against the iteration number and CPU time. From these figures, we can find that for dense matrices, i.e., the matrices with i.i.d. $N(0, 1)$ entries, the performances of the GRMK and GRK methods are almost the same; for sparse matrices, the GRMK method outperforms the GRK method in CPU time;  for all the cases, the GMBK and GDBK methods have almost the same performance.


\begin{figure}[ht]
 \begin{center}
\includegraphics [height=4.5cm,width=10cm]{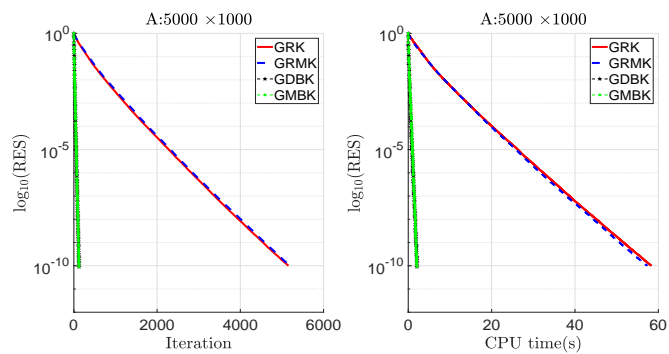}
 \end{center}
\caption{Performance of the methods on $A \in R^{ 5000 \times 1000}$ with i.i.d. $N(0, 1)$ entries.}\label{fig1}
\end{figure}
\begin{figure}[ht]
 \begin{center}
\includegraphics [height=4.5cm,width=10cm ]{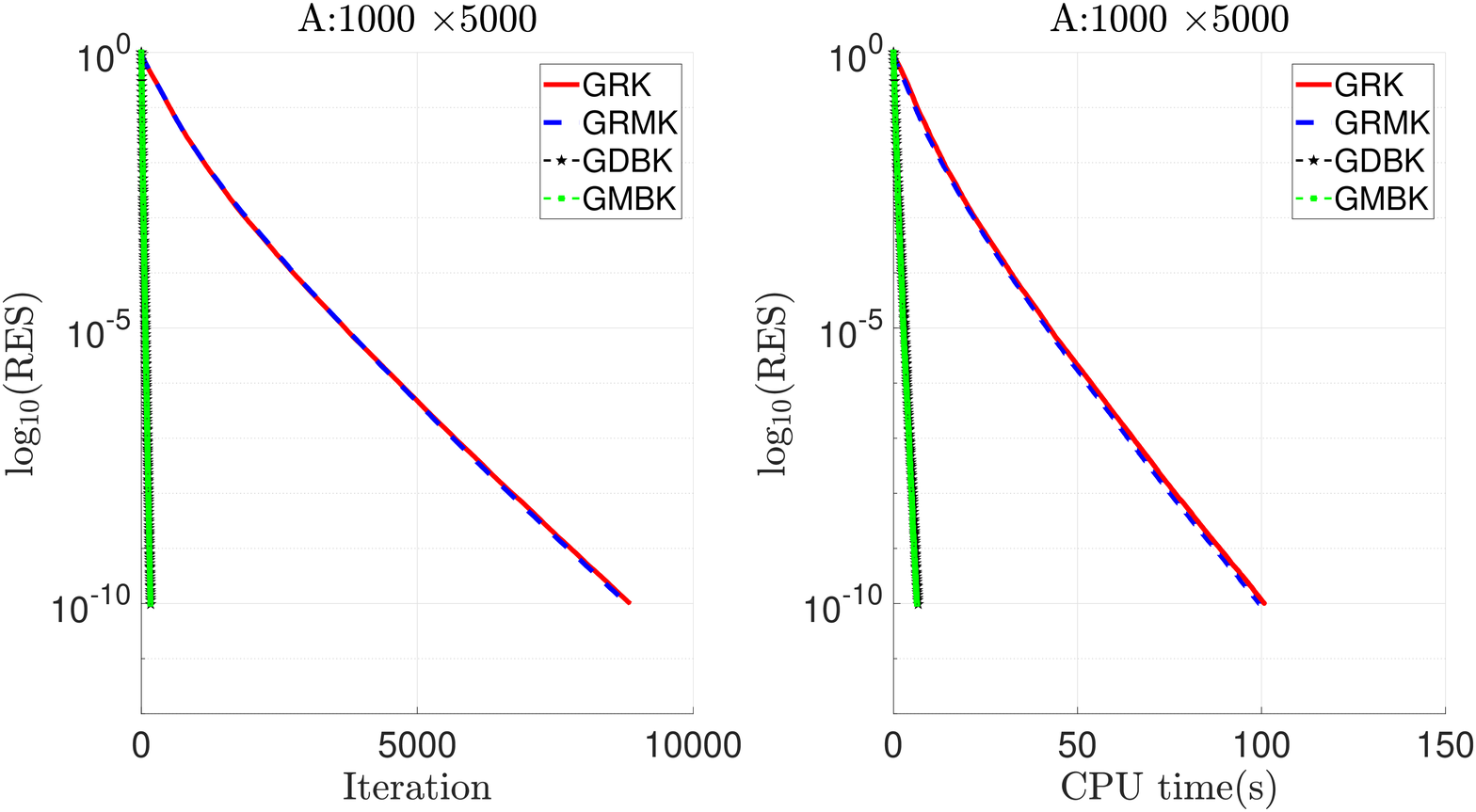}
 \end{center}
\caption{Performance of the methods on $A \in R^{1000 \times 5000}$
with i.i.d. $N(0, 1)$ entries.}\label{fig2}
\end{figure}
\begin{figure}[ht]
 \begin{center}
\includegraphics [height=4.5cm,width=10cm]{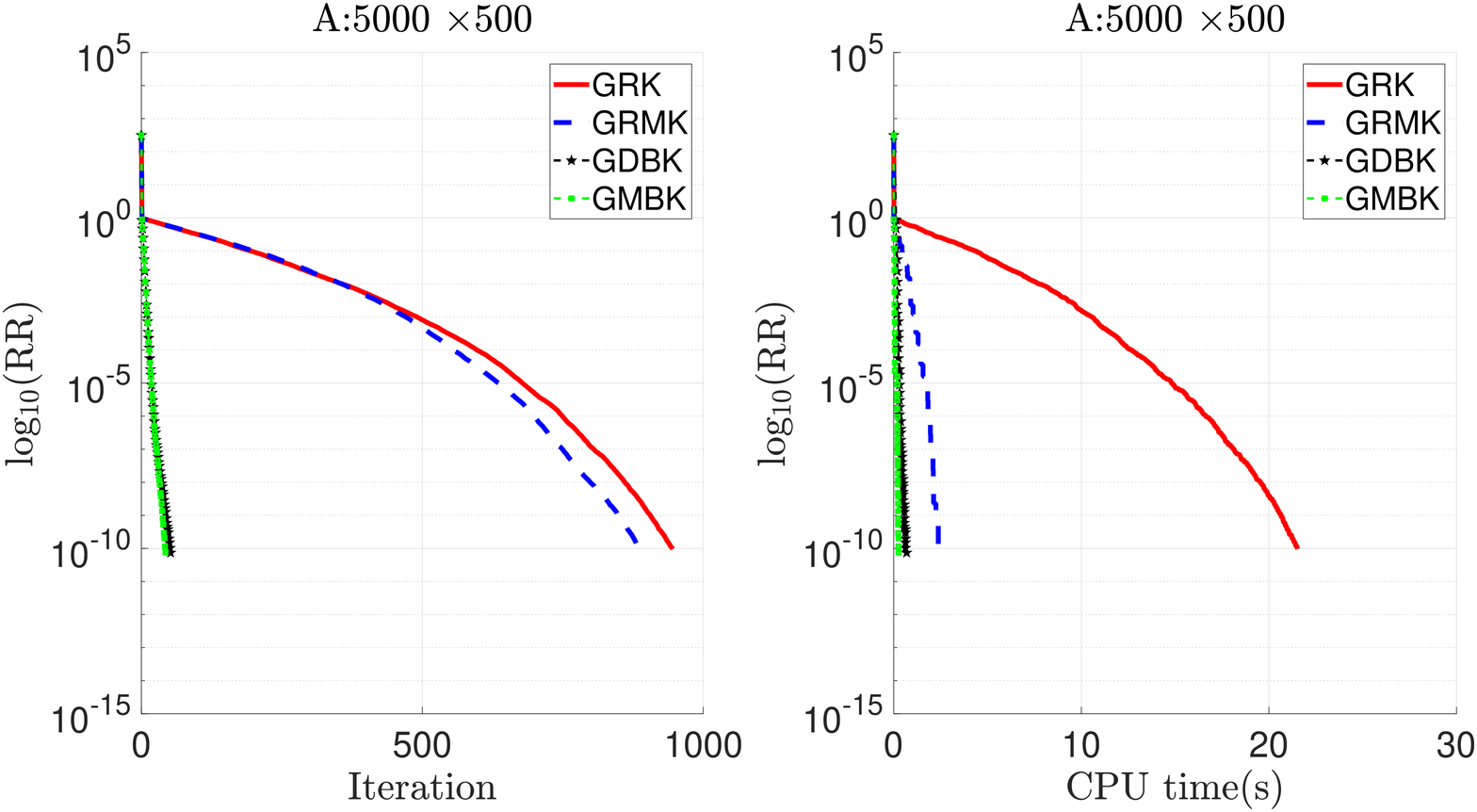}
 \end{center}
\caption{Performance of the methods on $A$ 
generated by the MATLAB function \texttt{sprandn(5000,500,0.2,0.8)}.}\label{fig3}
\end{figure}
\begin{figure}[ht]
 \begin{center}
\includegraphics [height=4.5cm,width=10cm ]{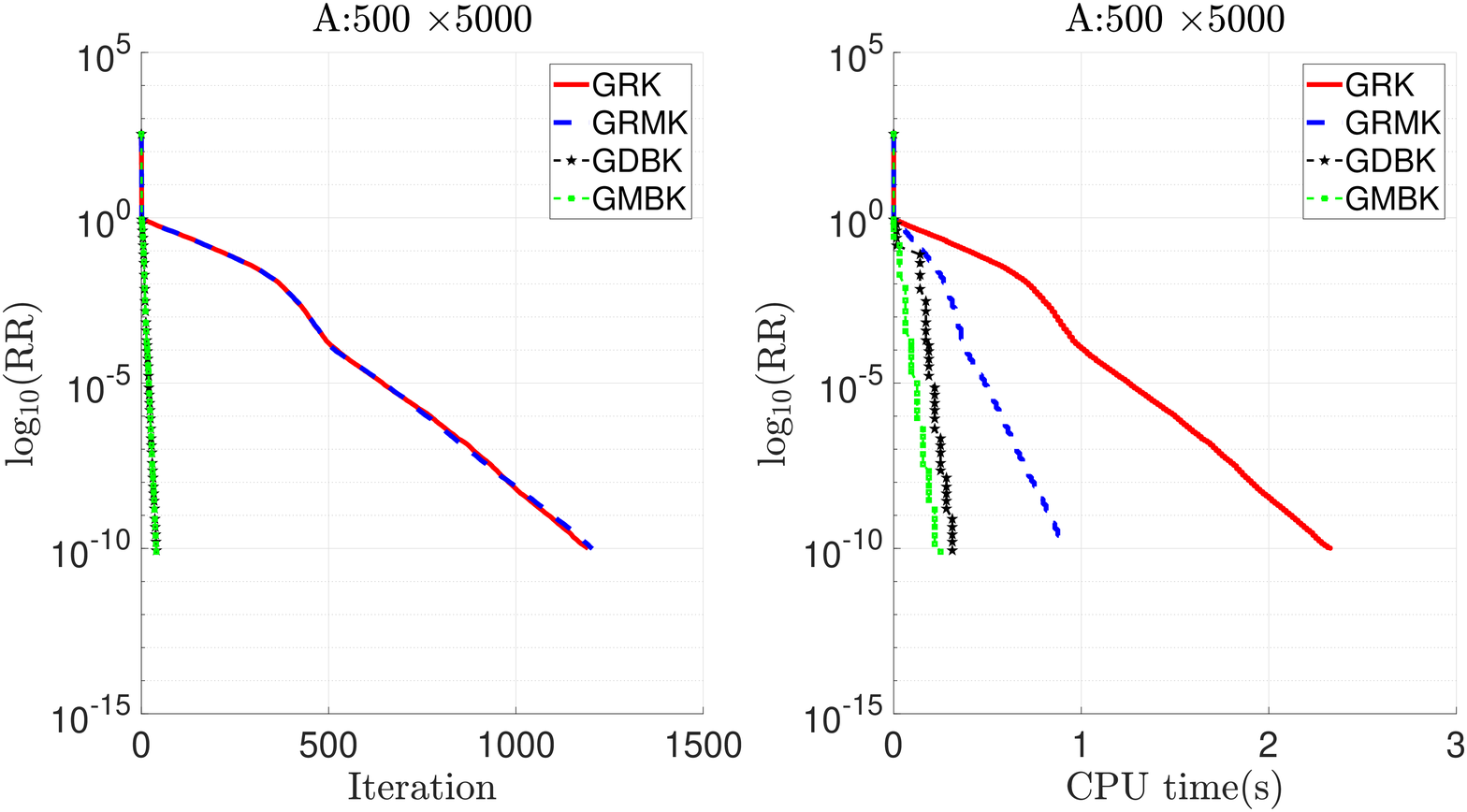}
 \end{center}
\caption{Performance of the methods on $A$ 
generated by the MATLAB function \texttt{sprandn(500,5000,0.2,0.8)}.}\label{fig4}
\end{figure}
\begin{figure}[ht]
 \begin{center}
\includegraphics [height=4.5cm,width=10cm]{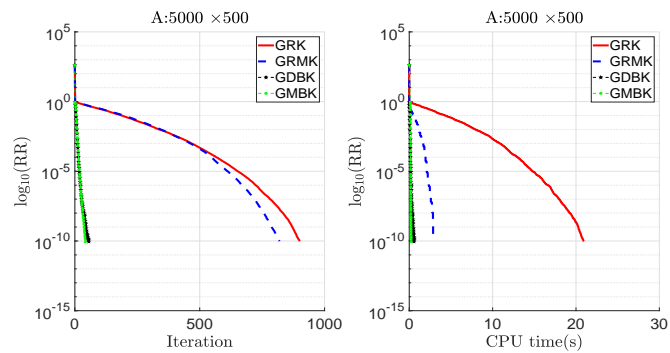}
 \end{center}
\caption{Performance of the methods on $A$ 
generated by the MATLAB function \texttt{sprand(5000,500,0.2,0.8)}.}\label{fig5}
\end{figure}
\begin{figure}[ht]
 \begin{center}
\includegraphics [height=4.5cm,width=10cm ]{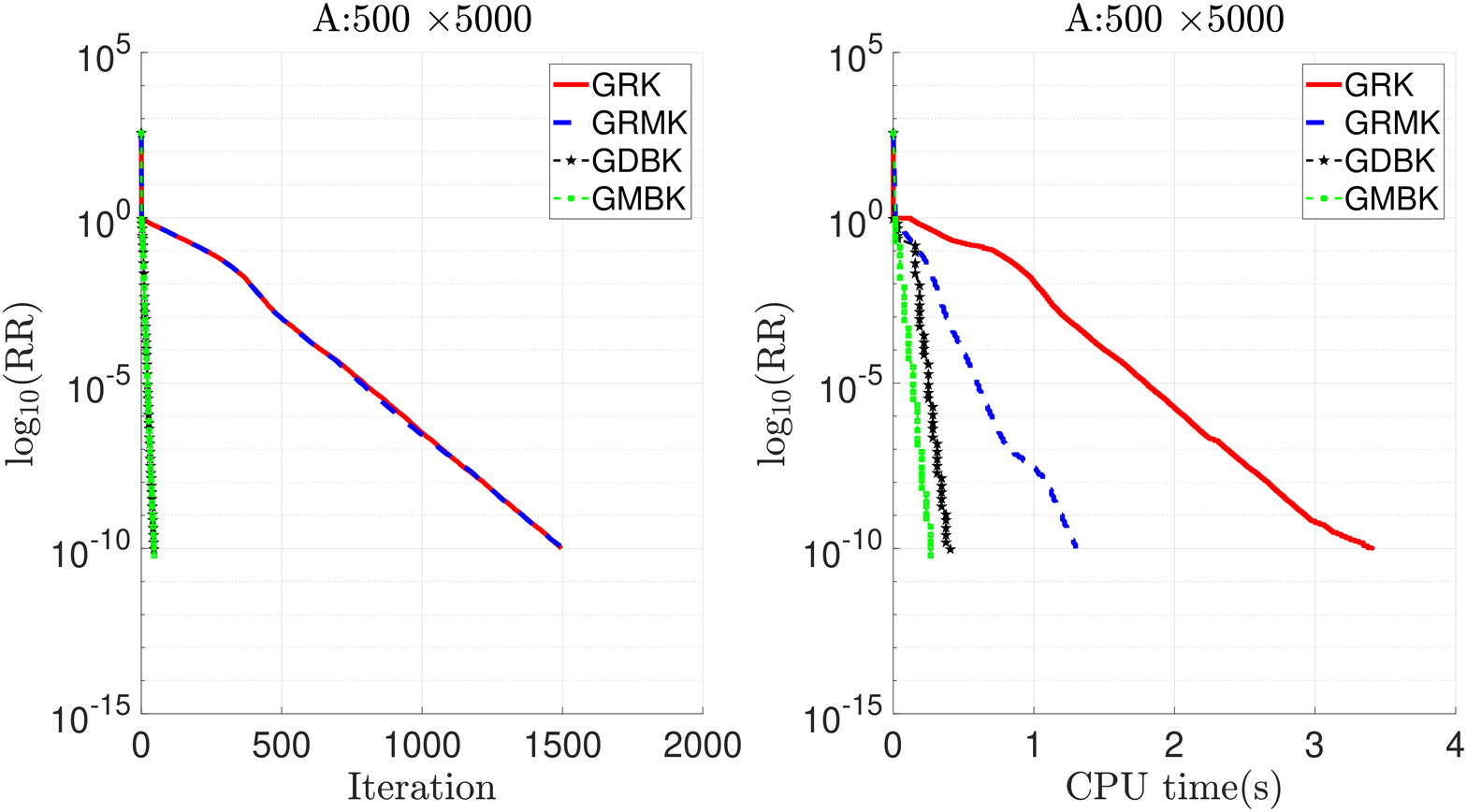}
 \end{center}
\caption{Performance of the methods on $A$ 
generated by the MATLAB function \texttt{sprand(500,5000,0.2,0.8)}.}\label{fig6}
\end{figure}

We also compare the performance of the methods 
on a real-world matrix, \textbf{mk9-b3}, taken from \cite{Davis2011}. Numerical results are reported in Figures \ref{fig7}--\ref{fig8}, 
which show the similar results obtained from the above experiments on dense matrices. That is,  the GRMK method and its block version have almost the same performance as the GRK method and its block version. 

\begin{figure}[ht]
 \begin{center}
\includegraphics [height=4.5cm,width=10cm]{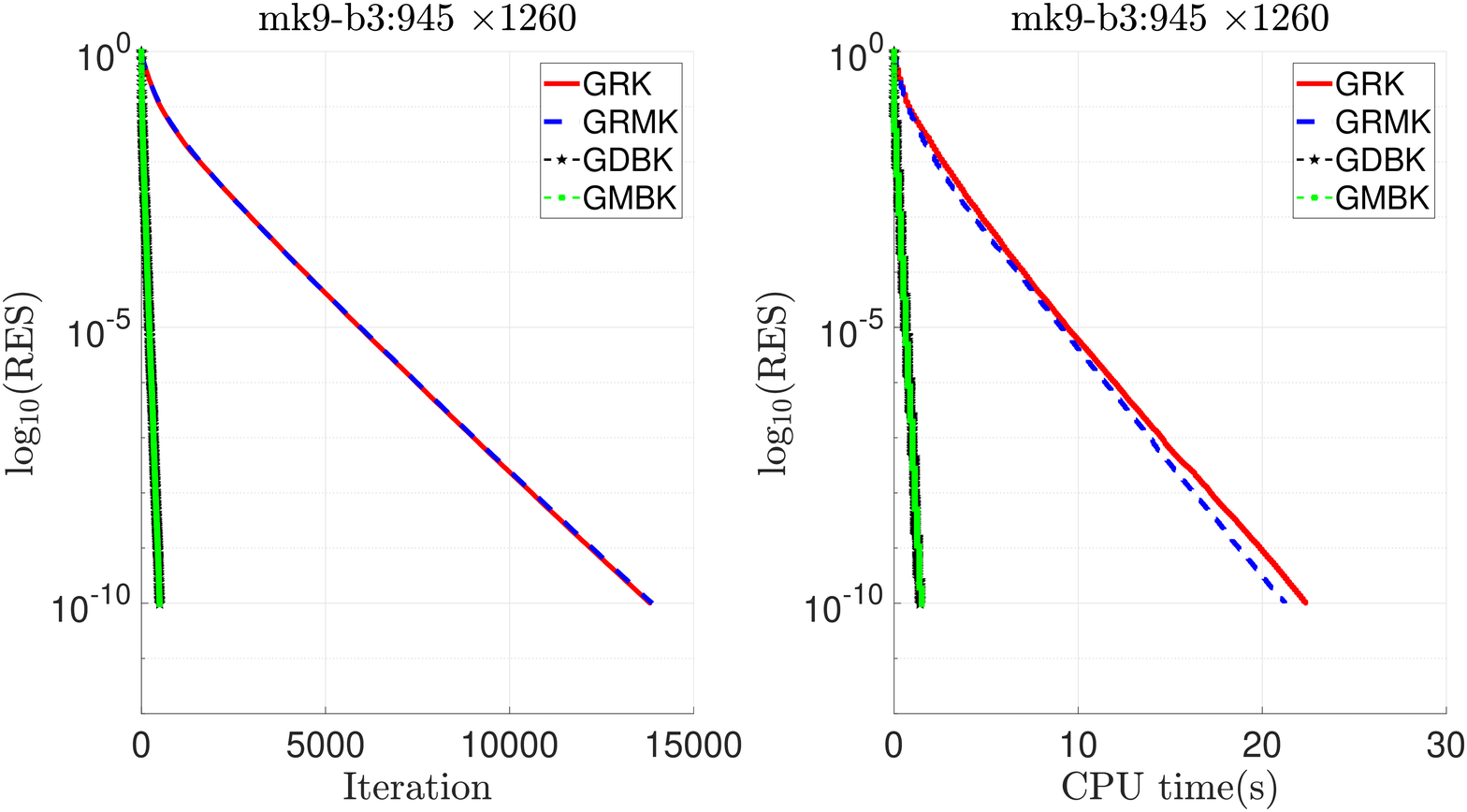}
 \end{center}
\caption{Performance of the methods on a real-world
matrix, \textbf{mk9-b3} ($945 \times 1260$).}\label{fig7}
\end{figure}
\begin{figure}[ht]
 \begin{center}
\includegraphics [height=4.5cm,width=10cm ]{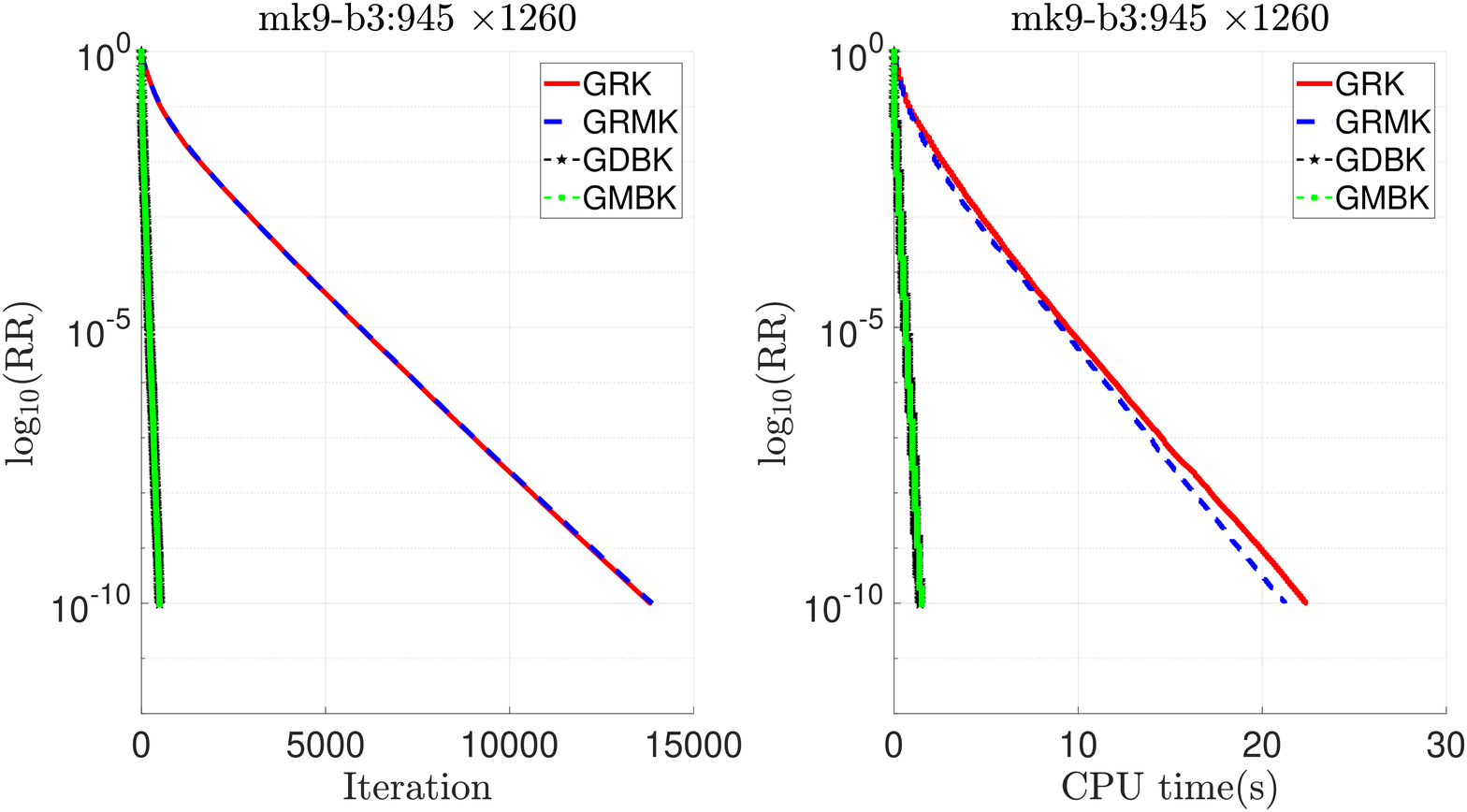}
 \end{center}
\caption{Performance of the methods on a real-world
matrix, \textbf{mk9-b3} ($945 \times 1260$).}\label{fig8}
\end{figure}

 \clearpage

\bibliography{mybibfile}

\end{document}